\documentclass[11pt,leqno]{article}
\usepackage[margin=1in]{geometry} 
\usepackage{amssymb,amsfonts,amsmath,bbm,mathrsfs,stmaryrd,mathtools}
\usepackage{xcolor}
\usepackage{url}

\usepackage{extarrows}

\usepackage[shortlabels]{enumitem}
\usepackage{tensor}

\usepackage{xr}
\usepackage[T1]{fontenc}

\usepackage[colorlinks,
linkcolor=black!75!red,
citecolor=blue,
pdftitle={},
pdfproducer={pdfLaTeX},
pdfpagemode=None,
bookmarksopen=true,
bookmarksnumbered=true,
backref=page]{hyperref}

\usepackage{tikz}
\usetikzlibrary{arrows,calc,decorations.pathreplacing,decorations.markings,decorations.shapes,intersections,shapes.geometric,through,fit,shapes.symbols,positioning,decorations.pathmorphing}

\usepackage{braket}

\usepackage[amsmath,thmmarks,hyperref]{ntheorem}
\usepackage{cleveref}

\newcommand\nthalias[1]{\AddToHook{env/#1/begin}{\crefalias{lemma}{#1}}}

\nthalias{definition}
\nthalias{example}
\nthalias{examples}
\nthalias{remark}
\nthalias{remarks}
\nthalias{convention}
\nthalias{notation}
\nthalias{construction}
\nthalias{sketch}
\nthalias{theoremN}
\nthalias{propositionN}
\nthalias{corollaryN}
\nthalias{lemma}
\nthalias{proposition}
\nthalias{corollary}
\nthalias{theorem}
\nthalias{conjecture}
\nthalias{question}
\nthalias{assumption}
\nthalias{recollection}

\creflabelformat{enumi}{#2#1#3}

\crefname{section}{Section}{Sections}
\crefformat{section}{#2Section~#1#3} 
\Crefformat{section}{#2Section~#1#3} 

\crefname{subsection}{\S}{\S\S}
\AtBeginDocument{%
  \crefformat{subsection}{#2\S#1#3}%
  \Crefformat{subsection}{#2\S#1#3}%
}

\crefname{subsubsection}{\S}{\S\S}
\AtBeginDocument{%
  \crefformat{subsubsection}{#2\S#1#3}%
  \Crefformat{subsubsection}{#2\S#1#3}%
}

\theoremstyle{plain}

\newtheorem{lemma}{Lemma}[section]

\newtheorem{corollary}[lemma]{Corollary}
\newtheorem{theorem}[lemma]{Theorem}

\theoremstyle{plain}
\theoremnumbering{Alph}
\newtheorem{theoremN}{Theorem}

\theoremstyle{plain}
\theorembodyfont{\upshape}
\theoremsymbol{\ensuremath{\blacklozenge}}

\newtheorem{definition}[lemma]{Definition}
\newtheorem{example}[lemma]{Example}

\newtheorem{remark}[lemma]{Remark}
\newtheorem{remarks}[lemma]{Remarks}

\crefname{definition}{definition}{definitions}
\crefformat{definition}{#2definition~#1#3} 
\Crefformat{definition}{#2Definition~#1#3} 

\crefname{ex}{example}{examples}
\crefformat{example}{#2example~#1#3} 
\Crefformat{example}{#2Example~#1#3} 

\crefname{exs}{example}{examples}
\crefformat{examples}{#2example~#1#3} 
\Crefformat{examples}{#2Example~#1#3} 

\crefname{remark}{remark}{remarks}
\crefformat{remark}{#2remark~#1#3} 
\Crefformat{remark}{#2Remark~#1#3} 

\crefname{remarks}{remark}{remarks}
\crefformat{remarks}{#2remark~#1#3} 
\Crefformat{remarks}{#2Remark~#1#3} 

\crefname{convention}{convention}{conventions}
\crefformat{convention}{#2convention~#1#3} 
\Crefformat{convention}{#2Convention~#1#3} 

\crefname{notation}{notation}{notations}
\crefformat{notation}{#2notation~#1#3} 
\Crefformat{notation}{#2Notation~#1#3} 

\crefname{table}{table}{tables}
\crefformat{table}{#2table~#1#3} 
\Crefformat{table}{#2Table~#1#3}

\crefname{lemma}{lemma}{lemmas}
\crefformat{lemma}{#2lemma~#1#3} 
\Crefformat{lemma}{#2Lemma~#1#3} 

\crefname{proposition}{proposition}{propositions}
\crefformat{proposition}{#2proposition~#1#3} 
\Crefformat{proposition}{#2Proposition~#1#3} 

\crefname{propositionN}{proposition}{propositions}
\crefformat{propositionN}{#2proposition~#1#3} 
\Crefformat{propositionN}{#2Proposition~#1#3} 

\crefname{corollary}{corollary}{corollaries}
\crefformat{corollary}{#2corollary~#1#3} 
\Crefformat{corollary}{#2Corollary~#1#3} 

\crefname{corollaryN}{corollary}{corollaries}
\crefformat{corollaryN}{#2corollary~#1#3} 
\Crefformat{corollaryN}{#2Corollary~#1#3} 

\crefname{theorem}{theorem}{theorems}
\crefformat{theorem}{#2theorem~#1#3} 
\Crefformat{theorem}{#2Theorem~#1#3} 

\crefname{theoremN}{theorem}{theorems}
\crefformat{theoremN}{#2theorem~#1#3} 
\Crefformat{theoremN}{#2Theorem~#1#3} 

\crefname{enumi}{}{}
\crefformat{enumi}{#2#1#3}
\Crefformat{enumi}{#2#1#3}

\crefname{assumption}{assumption}{Assumptions}
\crefformat{assumption}{#2assumption~#1#3} 
\Crefformat{assumption}{#2Assumption~#1#3} 

\crefname{construction}{construction}{Constructions}
\crefformat{construction}{#2construction~#1#3} 
\Crefformat{construction}{#2Construction~#1#3} 

\crefname{sketch}{sketch}{Sketches}
\crefformat{sketch}{#2sketch~#1#3} 
\Crefformat{sketch}{#2Sketch~#1#3} 

\crefname{recollection}{recollection}{Recollectiones}
\crefformat{recollection}{#2recollection~#1#3} 
\Crefformat{recollection}{#2Recollection~#1#3} 

\crefname{question}{question}{Questions}
\crefformat{question}{#2question~#1#3} 
\Crefformat{question}{#2Question~#1#3} 

\crefname{equation}{}{}
\crefformat{equation}{(#2#1#3)} 
\Crefformat{equation}{(#2#1#3)}

\numberwithin{equation}{section}

\theoremstyle{nonumberplain}
\theoremsymbol{\ensuremath{\blacksquare}}

\newtheorem{proof}{Proof}
\newcommand\pf[1]{\newtheorem{#1}{Proof of \Cref{#1}}}

\newcommand\bC{{\mathbb C}}

\newcommand\bG{{\mathbb G}}

\newcommand\bS{{\mathbb S}}

\newcommand\bU{{\mathbb U}}

\newcommand\cC{{\mathcal C}}

\newcommand\cF{{\mathcal F}}

\newcommand\cK{{\mathcal K}}
\newcommand\cL{{\mathcal L}}

\newcommand\cO{{\mathcal O}}

\newcommand\cU{{\mathcal U}}

\newcommand\wt{\widetilde}

\DeclareMathOperator{\id}{id}

\DeclareMathOperator{\im}{\mathrm{im}}

\DeclareMathOperator{\spn}{\mathrm{span}}
\DeclareMathOperator{\spec}{\mathrm{spec}}

\newcommand{\cat}[1]{\textsc{#1}}

\newcommand{\qedhere}{\mbox{}\hfill\ensuremath{\blacksquare}}

\newcommand{\xrightarrowdbl}[2][]{%
  \xrightarrow[#1]{#2}\mathrel{\mkern-14mu}\rightarrow
}

\title{Banach manifolds of spectrally small quantum-group representations}
\author{Alexandru Chirvasitu}

\begin{document}

\date{}

\newcommand{\Addresses}{{
  \bigskip
  \footnotesize

  \textsc{Department of Mathematics, University at Buffalo}
  \par\nopagebreak
  \textsc{Buffalo, NY 14260-2900, USA}  
  \par\nopagebreak
  \textit{E-mail address}: \texttt{achirvas@buffalo.edu}


}}

\maketitle

\begin{abstract}
  We prove that finite-spectrum representations of compact quantum groups either in unital $C^*$-algebras $A$ or on Banach spaces $E$ exhibit the same Banach-space-modeled differential-geometric structure as their classical analogues: (a) they are Banach analytic manifolds; (b) locally homogeneous under conjugation by the pertinent Banach Lie group $U(A)$ or $GL(E)$; (c) with orbit maps fibering principally; (d) and hence with said orbit maps admitting local analytic splitting.

  We also identify the finite-spectrum unitary representations as precisely those that are norm-continuous in the appropriate sense when the compact quantum group has at least one classical point, again generalizing the classical parallel present in various forms in work of Kallman, Shtern and the author. 
\end{abstract}


\vspace{.5cm}



\section*{Introduction}

Part of the motivation for the present paper situates it in the context of wide-ranging phenomena to the effect that moduli spaces parametrizing representation-theoretic data attached to ``sufficiently semisimple'' objects have rich geometric structure and exhibit strong regularity properties. A sampling of the vast and multifaceted literature will reveal many instances concretizing the general pattern in various ways.
\begin{itemize}[wide]
\item Structuring spaces of projections/unitaries/partial isometries/orthogonal 1-sum projection $n$-tuples (and many other flavors of like objects) in Banach or $C^*$-algebras $B$ as (typically analytic) \emph{Banach manifolds} in the sense of \cite[\S 5.1]{bourb_vars_1-7} fits the mold, as said spaces are often interpretable as morphisms $A\to B$ for well-behaved Banach/$C^*$-algebras $A$. Examples include \cite{MR2048217,MR1239452,MR1051073} to list only a few.

\item More generally (and as just alluded to), spaces of Banach-algebra/$C^*$ morphisms $A\to B$ for amenable or finite-dimensional semisimple $A$ can often also be equipped with Banach-manifold structures and enjoy local homogeneity under the invertible/unitary group of $B$, decomposability as locally trivial fibrations resulting from that action, etc. See for instance \cite[Proposition 2.2]{MR1354040}, \cite[\S 4]{zbMATH07787455}, \cite[\S 8.2]{rnd_amnbl_2002} and their many cited references.

\item In like fashion (and most directly linked to this paper), for compact groups $\bG$ and unital $C^*$-algebras $A$, \cite{mart-proj} equips the spaces $\cat{Rep}_{\sigma}(\bG,A)$ of \emph{cocycle-twisted} $\bG$-representations in $A$ with Banach-manifold structures and studies their right local geometry. \cite[Theorems 3.2 and 3.7]{zbMATH07787455} extend that discussion, focusing this time on spaces of cocycles of a compact group $\bG$ valued in a \emph{Banach Lie group} \cite[\S 6]{upm_ban} $\bU$.
\end{itemize}

One aim, here, is to divorce at least part of the picture painted by the last bullet point of its ``classical baggage'' by extending analogous results to compact \emph{quantum} $\bG$ in the sense of \cite[Definition 2.1]{wor-cqg}, embodied here as their corresponding $C^*$-algebras $\cC(\bG)$ of continuous functions. The narrative runs in parallel in two adjacent but not-quite-coincident threads, with the following summary compressing and paraphrasing \Cref{th:unifmor2amfld,th:fin.spc.ban} for the purposes of the present introduction. $\bG$ is a compact quantum group throughout, with relevant background and terminology unwound subsequently.   

\begin{theoremN}\label{thn:ban.fib.both}
  For compact quantum $\bG$, \emph{finite-spectrum} $\bG$-representations \emph{in} a unital $C^*$-algebra $A$ (or \emph{on} a Banach space $E$)
  \begin{enumerate}[(1),wide]
  \item constitute analytic Banach manifolds;
  \item locally homogeneous under conjugation by the naturally acting group $\bU$: the unitary group $U(A)$ in one case, and $GL(E)$ in the other;
  \item with the resulting orbit maps
    \begin{equation*}
      \bU\ni g
      \xmapsto{\quad}
      \left(g\triangleright x\in \text{orbit of $x$}\right)
    \end{equation*}
    analytic \emph{principal fibrations} \cite[\S 6.2]{bourb_vars_1-7};

  \item so that in particular said orbit maps split locally and analytically: every $x$ has an open neighborhood $\cU\ni x$ supporting an analytic map
    \begin{equation*}
      \cU
      \ni x'
      \xmapsto{\quad}
      u_{x'}
      \in
      \bU
      ,\quad
      u_x=1,\quad \forall\left(x'\in \cU\right)\left(u_{x'}\triangleright x'=x\right).
    \end{equation*}
  \end{enumerate}
  \qedhere
\end{theoremN}

The finite-spectrum requirement requires explication, serving also to bridge the paper's two main sections. Note that the classical results \Cref{thn:ban.fib.both} extends all involve \emph{norm}-continuous compact-group representations $\bG\to GL(E)$ on Banach spaces $E$ rather than the broader, familiar \cite[\S 2]{rob} class of \emph{strongly} continuous ones:
\begin{equation*}
  \forall\left(v\in E\right)
  \left(\bG\ni g\xmapsto[\quad\text{continuous}\quad]{\quad}gv\in E\right).
\end{equation*}
Classically, that norm continuity is precisely characterized, representation-theoretically, by aforementioned \emph{spectrum finiteness}: the norm-continuous representations are exactly those with finitely many \emph{isotypic components} \cite[Theorem 4.1]{hm5} $\rho^{\alpha}$:
\begin{equation*}
  \begin{aligned}
    \bG
    &\xrightarrow[\quad\text{norm-continuous}\quad]{\quad}
      GL(E)
      \iff
      \rho=\bigoplus_{\alpha\in \mathrm{Irr}(\bG)}\rho^{\alpha}
      \text{ is finite}\\
    \rho^{\alpha}
    &:=\text{largest direct sum of copies of $\alpha\in \mathrm{Irr}(\bG)$}.
  \end{aligned}  
\end{equation*}
See for instance \cite[Corollary 2]{zbMATH05628052} for this precise statement, \cite[Theorem 11]{klm_unif} for a unitary precursor in the presence of additional tameness constraints on $\bG$, and \cite[Theorem 3.10]{Chirvasitu2026JNCG604} for multiple equivalent characterizations.

Compact quantum groups admit isotypic decompositions to their Banach-space/unitary representations (with background recalled in \Cref{se:nc.fin.spec}), hence our casting the finite-spectrum constraint, below, as an appropriate substitute for norm continuity.

There are also, however, more direct analogues of (quantum) norm continuity: for a unitary representation of the compact quantum group $\bG$, cast \Cref{eq:rep.gh} as a unitary element $U\in M(\cC(\bG)\underline{\otimes} \cK(H))$, for instance, one can simply require that $U$ belong to the spatial $C^*$ tensor product $\cC(\bG)\underline{\otimes} \cL(H)$. It becomes natural, in context, to compare the two competing concepts of norm continuity; \Cref{th:equivunif} proves them equivalent assuming $\bG$ \emph{has a classical point}, i.e. the quantum function algebra $\cC(\bG)$ has a multiplicative state.

\begin{theorem}\label{th:equivunif}
  Let $\bG$ be a compact quantum group with at least one classical point. The following conditions on a unitary $\bG$-representation
  \begin{equation}\label{eq:gurep}
    U\in M(\cC(\bG)\underline{\otimes}\cK(H))
    ,\quad
    H\text{ a Hilbert space}
  \end{equation}
  are equivalent.
  \begin{enumerate}[(a),wide]
  \item\label{item:th:equivunif:finisot} $U$ has finitely many isotypic components. 

  \item\label{item:th:equivunif:unif} $U$ is uniformly continuous in the sense that $U\in \cC(\bG)\underline{\otimes}\cL(H)$.
  \end{enumerate}
\end{theorem}


\section{Norm continuity for CQG representations and finite spectra}\label{se:nc.fin.spec}

We take for granted the familiar background on \emph{compact quantum groups} as covered, for instance, in \cite{kt_qg-surv-1,NeTu13,tim,wor-cqg} and like sources (with focused references as needed). Specifically, every such is a virtual object $\bG$ dual to a number of function-algebra-type incarnations, circumstances determining the appropriate choice:
\begin{itemize}[wide]
\item A \emph{CQG algebra} \cite[Definition 2.2]{dk_cqg} $\cO(\bG)$, i.e. a complex (\emph{cosemisimple} \cite[Definition 2.4.1]{mont}) Hopf $*$-algebra whose simple comodules are \emph{unitarizable}: admitting a right-linear inner product $\Braket{-\mid-}$ invariant under the $\cO(\bG)$-coaction, in the sense \cite[(1.8)]{dk_cqg} that
  \begin{equation*}
    \forall\left(v,w\in V\right)
    \left(
      \Braket{v\mid w_0}Sw_1
      =
      \Braket{v_0\mid w}v_1^*
    \right)
  \end{equation*}
  in \emph{Sweedler notation} \cite[Notation 1.4.2 and post Definition 1.6.2]{mont} for the coaction
  \begin{equation*}
    V\ni v
    \xmapsto{\quad}
    v_0\otimes v_1
    \in V\otimes \cO(\bG).
  \end{equation*}
  $\cO(\bG)$ is meant as analogous to the algebra of \emph{representative functions} \cite[Definition 3.3]{hm5} on an ordinary compact group. 

\item The universal $C^*$ completion $\cC_u(\bG)$ \cite[post Corollary 1.7.5]{NeTu13} of $\cO(\bG)$.

\item The \emph{reduced} version $\cC_r(\bG)$, image of the GNS representation of $\cO(\bG)$ with respect to its \emph{Haar integral (or state)} $h_{\bG}\in \cO(\bG)^*$ \cite[\S 1.2]{NeTu13}.

\item Intermediate avatars between $\cC_u$ and $\cC_r$ generally (leading to the $C^*$ version of the definition in its broadest terms): $C^*$-algebras $A=\cC(\bG)$ sandwiched as
  \begin{equation*}
    \cC_u(\bG)
    \xrightarrowdbl{\quad}
    A
    \xrightarrowdbl{\quad}
    \cC_r(\bG)
  \end{equation*}
  so as to inherit $\cC_u$'s coassociative $C^*$ morphism $A\xrightarrow{\Delta} A\underline{\otimes} A$ (minimal tensor product). 
\end{itemize}

We will also be referring extensively to CQG \emph{unitary representations} on Hilbert spaces as conveniently recalled, say in \cite[\S 2.2]{dsv}: unitaries
\begin{equation*}
  \begin{aligned}
    U
    \in
    M\left(\cC_r(\bG)\underline{\otimes}\cK(H)\right)
    &\subseteq
      L^{\infty}(\bG)\overline{\otimes} \cL(H)
      \quad
      \left(\text{$W^*$ tensor product \cite[Definition IV.5.1]{tak1}}\right)\\
    \left(\Delta\otimes \id\right)U
    &=
      U_{13}U_{23}
      \quad\left(\text{\emph{leg notation} \cite[\S 2.1.2]{tim}}\right),
  \end{aligned}
\end{equation*}
with $M(-)$ and $\cK(-)$ denoting the respective \emph{multiplier algebra} \cite[Proposition 2.1.3]{wo} and compact-operator algebra and $L^{\infty}(\bG)\ge \cC_r(\bG)$ the von Neumann closure of the GNS representation attached to $h_{\bG}$.

Below, we frequently cast $\bG$ as dual to such a generic Hopf $C^*$-algebra $\cC(\bG)$ (rather than specifically its reduced or universal versions). Set, in that framework,
\begin{equation}\label{eq:rep.gh}
  \cat{Rep}(\bG,H)
  :=
  \left\{
    U\in UM\left(\cC(\bG)\underline{\otimes}\cK(H)\right)
    \ :\
    \left(\Delta\otimes \id\right)U
    =
    U_{13}U_{23}
  \right\}.
\end{equation}
The decomposition
\begin{equation*}
  \begin{aligned}
    \cO(\bG)
    &\cong
      \bigoplus_{\alpha\in \mathrm{Irr}(\bG)}C^{\alpha}
      =
      \bigoplus_{\alpha}\spn\{u_{ij}^{\alpha},\ 1\le i,j\le \dim V^{\alpha}\}\\
    \mathrm{Irr}(\bG)
    &:=
      \left(
      \text{classes of simple $\cO(\bG)$-comodules }
      V^{\alpha}
      \xrightarrow{\quad\rho^{\alpha}\quad}
      V^{\alpha}\otimes \cO(\bG)
      \right)
  \end{aligned}
\end{equation*}
(the $u^{\alpha}_{ij}$ being the usual matrix coefficients of $\bG$ \cite[Proposition 4.7 and surrounding discussion]{wor}) affords \cite[Theorem 1.5]{podl_symm}
\begin{itemize}[wide]
\item functionals $\psi^{\alpha}\in \cO(\bG)^*$ respectively annihilating all but the diagonal $u^{\alpha}_{ii}$, sent to 1 instead;
\item whence also (\emph{spectral}) projections
  \begin{equation}\label{eq:spc.proj}
    P^{\alpha}
    =
    P_U^{\alpha}
    :=
    \left(\psi^{\alpha}\otimes\id\right)U^*
    \in
    \cL(H),
  \end{equation}
  each having the \emph{$\alpha$-isotypic component} $H^{\alpha}\le H$ of the representation $U$ as its image: the largest $\bG$-invariant subspace of $H$ on which the representation decomposes as a sum of copies of $\rho^{\alpha}$. 
\end{itemize}
The \emph{spectrum} $\spec U$ is then the collection of $\alpha$ producing non-zero $P^{\alpha}$. 

\pf{th:equivunif}
\begin{th:equivunif}
  We focus on the interesting implication \Cref{item:th:equivunif:unif} $\Rightarrow$ \Cref{item:th:equivunif:finisot}, the converse being straightforward (see also \cite[Remark 3.9]{Chirvasitu2026JNCG604}).
  
  The meaning of the hypothesis is that there is some continuous multiplicative $*$-morphism $\cC(\bG)\to \bC$. As explained in \cite[paragraph preceding Theorem 2.2]{zbMATH07377294}, it follows that the \emph{counit} \cite[post Definition 1.6.1]{NeTu13} of $\bG$ extends continuously across $\cC(\bG)$.
  
  \cite[Theorem 3.16]{Chirvasitu2026JNCG604} equates \Cref{item:th:equivunif:finisot} to the extensibility of
  \begin{equation}\label{eq:kconjbyu}
    \cK(H)
    \ni
    T
    \xmapsto{\quad}
    U^*(1\otimes T)U
    \in
    M(\cC(\bG)\underline{\otimes}\cK(H))
  \end{equation}
  to an \emph{action} \cite[Definition 1.4]{podl_symm}
  \begin{equation}\label{eq:lconjbyu}
    \cL(H)
    \xrightarrow{\quad\rho\quad}
    \cC(\bG)\underline{\otimes}\cL(H).
  \end{equation}
  Given the assumed \Cref{item:th:equivunif:unif}, the selfsame formula \Cref{eq:kconjbyu} certainly provides a coassociative map \Cref{eq:lconjbyu}; the one outstanding ingredient is the density constraint in \cite[Definition 1.4 b)]{podl_symm}:
  \begin{equation*}
    \overline{\spn\left\{(y\otimes 1)\rho(x)
        =
        (y\otimes 1)U^*(1\otimes x)U
        \ |\ x\in \cL(H),\ y\in \cC(\bG)\right\}}^{\|\cdot\|}
    =
    \cC(\bG)\underline{\otimes}\cL(H). 
  \end{equation*}
  In the presence of a counit $\cC(\bG)\xrightarrow{\varepsilon}\bC$ that density condition is equivalent (as follows immediately from \cite[Theorem 6.3]{MR1629723}, say) to
  \begin{equation*}
    \cL(H)
    \xrightarrow{\quad(\varepsilon\otimes \id)\rho = \id\quad}
    \cL(H).
  \end{equation*}
  That condition holds for $\cK(H)$, hence the (unique \cite[Lemma 3.2.4 and Theorem 15.2.12]{wo}) extension
  \begin{equation*}
    \begin{tikzpicture}[>=stealth,auto,baseline=(current  bounding  box.center)]
      \path[anchor=base] 
      (0,0) node (l) {$\cK(H)$}
      +(3,.5) node (u) {$\cL(H)\cong M(\cK(H))$}      
      +(6,0) node (r) {$\cL(H)$}
      ;
      \draw[right hook->] (l) to[bend left=6] node[pos=.5,auto] {$\scriptstyle $} (u);
      \draw[->] (u) to[bend left=6] node[pos=.5,auto] {$\scriptstyle \id$} (r);
      \draw[right hook->] (l) to[bend right=6] node[pos=.5,auto,swap] {$\scriptstyle (\varepsilon\otimes \id)\rho$} (r);
    \end{tikzpicture}
  \end{equation*}
  to the multiplier algebra.
\end{th:equivunif}


\Cref{th:equivunif} recovers the classical characterization of norm-continuous unitary representations of a (plain, non-quantum) compact group as precisely those with finite spectrum (\cite[Corollary 2]{zbMATH05628052} in general, \cite[Theorem 3.10]{Chirvasitu2026JNCG604} for a number of alternative characterizations, or \cite[Theorem 11]{klm_unif} for unitary representations of \emph{locally} compact connected $2^{nd}$ countable groups): naturally, commutative $\cC(\bG)$ always have counits.

The discussion extends to (what we will be referring to as) \emph{uniform} (or \emph{uniformly continuous}) $\bG$-representations in arbitrary $C^*$-algebras $A$: patterning the definition on that of a uniform unitary $\bG$-representation (operative in \Cref{th:equivunif}),
\begin{equation*}
  \cat{Rep}(\bG,A)
  :=
  \left\{U\in \text{unitary group }U\left(\cC(\bG)\underline{\otimes}A\right)\ :\ \left(\Delta\otimes \id\right)U = U_{13}U_{23}\right\}
\end{equation*}
(cf. \cite[p.292]{kus-univ}). Spectra are meaningful in this setting as well.

\begin{definition}\label{def:spec.rep.g.a}
  For $U\in \cat{Rep}(\bG,A)$,
  \begin{equation*}
    \spec U
    :=
    \spec \wt{U}
    ,\quad
    \begin{aligned}
      \wt U
      &:=(\id\otimes \iota)U\in U\left(\cC(\bG)\underline{\otimes}\cL(H)\right)\\
      A
      &\lhook\joinrel\xrightarrow[\quad \text{$C^*$ embedding}\quad]{\quad\iota\quad}\cL(H).
    \end{aligned}  
  \end{equation*}
  Plainly, \Cref{eq:spc.proj} (with $\wt{U}$ in place of $U$) vanishes for one $\iota$ if and only if it does for any other, hence the definition's independence of embedding. 
\end{definition}

The proof of \Cref{th:equivunif} now also yields:

\begin{theorem}\label{th:isautofinsp}
  For a compact quantum group $\bG$ with a classical point and a unital $C^*$-algebra $A$ a representation $U\in \cC(\bG)\underline{\otimes}A$ has finite spectrum in any of the following mutually equivalent senses:
  \begin{enumerate}[(a),wide]
  \item The unitary $\bG$-representation
    \begin{equation*}
      (\id\otimes\psi)U\in \cC(\bG)\underline{\otimes}\cL(H)
    \end{equation*}
    has finite spectrum for any $*$-representation $A\xrightarrow{\psi}\cL(H)$.

  \item As above, for \emph{one} faithful $A\xrightarrow{\psi}\cL(H)$.
    
  \item The conjugation $\bG$-action $U^*(1\otimes\bullet)U$ of $\bG$ on $A$ induced by $\bG$ has finite spectrum.  \qedhere
  \end{enumerate}
\end{theorem}

\begin{remark}\label{re:rec.cls}
  \Cref{th:equivunif,th:isautofinsp} in particular retrieve their classical $\bG$ counterparts by means rather different than what one typically encounter in the literature: the aforementioned \cite[Corollary 2]{zbMATH05628052}, \cite[Theorem 3.10]{Chirvasitu2026JNCG604} and analogues tend to rely on structural properties of locally compact groups and/or Banach Lie groups highly specific to their classical setup. 
\end{remark}

\section{Manifolds of quantum-group representations}\label{se:repmfld}

Recall from \cite[\S 3.2]{mart-proj} that for any compact group $\bG$, unital $C^*$-algebra $A$ and continuous \emph{2-cocycle} \cite[\S 1.1]{mart-proj} $\bG\times \bG\xrightarrow{\sigma}\bS^1$ the space $\cat{Rep}_{\sigma}(\bG,A)$ of representations $\bG\to A$ twisted by $\sigma$ (\emph{$\sigma$-representations} in \cite[\S 1.4]{mart-proj}, compatible also with the terminology of \cite[\S VII.2]{vrd}, say) is a real analytic \emph{Banach submanifold} of $\cC(\bG,A)$; the latter are to be understood in the sense of \cite[\S II.2]{lang-fund} or \cite[\S 5.8.3]{bourb_vars_1-7}.

Given that the uniformly continuous Banach-space representations of a compact group are precisely those with finite spectrum, for compact \emph{quantum} $\bG$ it is also the finite-spectrum unitary representations of \Cref{th:isautofinsp} that are best behaved with respect to affording smooth Banach-manifold parametrization. 

\Cref{th:unifmor2amfld} formalizes the last remark, and requires some preliminary observations concerning continuity properties of CQG-representation spectra. 

\begin{lemma}\label{le:spc.cont}
  Let $\bG$ be a compact quantum group and $H$ a Hilbert space. 
  \begin{enumerate}[(1),wide]
  \item\label{item:le:spc.cont:h} The map
    \begin{equation*}
      \cat{Rep}(\bG,H)
      \ni U
      \xmapsto{\quad}
      \spec U
      \in 2^{\mathrm{Irr}(\bG)}=\{0,1\}^{\mathrm{Irr}(\bG)}.
    \end{equation*}
    is continuous for the uniform topology on the domain and the compact-open topology on the codomain.

  \item\label{item:le:spc.cont:a} In particular, for any unital $C^*$-algebra $A$ the map
    \begin{equation*}
      \cat{Rep}(\bG,A)
      \ni U
      \xmapsto{\quad}
      \spec U     
      \in 2^{\mathrm{Irr}(\bG)}
    \end{equation*}
    enjoys the same continuity property, and restricts to a locally constant map on the open subset
    \begin{equation*}
      \cat{Rep}_{\cat{fs}}(\bG,A)
      :=
      \left\{U\in \cat{Rep}(\bG,A)\ :\ |\spec U|<\infty\right\}
      \subseteq
      \cat{Rep}(\bG,A).
    \end{equation*}
  \end{enumerate}
\end{lemma}
\begin{proof}
  \begin{enumerate}[label={},wide]
  \item\textbf{\Cref{item:le:spc.cont:h}} For each $\alpha\in \mathrm{Irr}(\bG)$ the map $U\mapsto P^{\alpha}_U$ (see \Cref{eq:spc.proj}) is norm-norm continuous, and sufficiently (mutually) close projections are conjugate \cite[Proposition 5.2.6]{wo} (so do or do not vanish simultaneously). 
    
  \item\textbf{\Cref{item:le:spc.cont:a}} The twofold claim perhaps needing some verification is openness+local constancy. By \Cref{item:le:spc.cont:h} (and its proof), all representations $U'$ sufficiently close to $U\in \cat{Rep}_{\cat{fs}}(\bG,A)$ will have its $P^{\alpha}_{U'}$, $\alpha\in \spec U$ respectively norm-close to $P^{\alpha}_U$. We thus have close
    \begin{equation*}
      \left(P^{\alpha}_{U'}\right)_{\alpha\in \spec U}
      \ 
      \sim
      \ 
      \left(P^{\alpha}_{U}\right)_{\alpha\in \spec U}
      \ \text{in}\ C^*\left(\bC^{|\spec U|}, A\right).
    \end{equation*}
    These will be mutual conjugates under the unitary group $U(A)$ (e.g. \cite[Theorem 4.1]{zbMATH07787455}, \cite[Proposition 2.2]{MR1354040} for $A=\cL(H)$, etc.), so that
    \begin{equation*}
      \sum_{\alpha\in \spec U}P^{\alpha}_{U'}
      =
      \sum_{\alpha\in \spec U}P^{\alpha}_{U}
      =
      1.
    \end{equation*}
    In particular, $\spec U'=\spec U$ (finite, by assumption). 
  \end{enumerate}
\end{proof}

\begin{remarks}\label{res:wtops}
  \begin{enumerate}[(1),wide]
  \item\label{item:res:wtops:tops} There are several natural ways to topologize unitary $\bG$-representations. \cite[p.85]{dsv}, for instance, restricts the \emph{strict} topology \cite[\S 3.12.17]{ped-aut} of $M(-)$ to its unitary subgroup. There are also available all of the usual weaker-than-norm locally convex topologies one typically equips $\cL(H)$ spaces with (recalled, say, in \cite[\S II.2]{tak1} or \cite[\S 2.1.7]{ped-aut}, etc.). These all agree on the unitary group $U\left(L^{\infty}(\bG)\overline{\otimes} \cL(H)\right)$ (and equip it with a topological-group structure):
    \begin{itemize}[wide]
    \item weak=weak$_{\sigma}$, strong=strong$_{\sigma}$ and strong$^*$=strong$^*_{\sigma}$ on bounded sets as noted in the selfsame \cite[\S 2.1.7]{ped-aut};
    \item weak=strong (on isometries, so also unitaries) by \cite[Solution 20]{hal_hspb_2e_1982}, and strong=strong$^*$ on normal operators in any case \cite[Solution 116]{hal_hspb_2e_1982};
    \item per \cite[Theorem II.7]{zbMATH03253462}, on the unit ball strong$^*$=\emph{Mackey} \cite[\S 36]{trev_tvs}, the strongest locally convex topology affording the same space of continuous functionals as weak$_{\sigma}$;

    \item and finally, the unitaries constitute a topological group under these by \cite[Proposition 18.4.15]{hjjm_bdle}.
    \end{itemize}

  \item\label{item:res:wtops:cls.exs} There is no reason why \Cref{le:spc.cont} would hold for the weaker topologies flagged in \Cref{item:res:wtops:tops} above: it is a simple matter (\Cref{ex:cls.str.conv}) to produce finite-spectrum nets of classical compact-group representations converging in the strict topology (so also weakly/strongly/Mackey/etc.) without the corresponding compact-open spectrum convergence.
  \end{enumerate}  
\end{remarks}

\begin{example}\label{ex:cls.str.conv}  
  On an infinite-dimensional Hilbert space $H$, consider an increasing net
  \begin{equation*}
    \cL(H)
    \ni
    \left(\text{projection }P_{\lambda}\right)
    \xrightarrow[\quad\lambda\quad]{\quad\text{strong$^*$}\quad}
    1
    \in \cL(H)
    ,\quad \dim P_{\lambda}H^{\perp}=\infty
  \end{equation*}
  and have any compact group $\bG$ act trivially on $P_{\lambda}H$ and $\alpha$-isotypically on $P_{\lambda}H^{\perp}$ for some $1\ne \alpha\in \mathrm{Irr}(\bG)$. The corresponding
  \begin{equation}\label{eq:ulambda}
    \begin{aligned}
      U_{\lambda}
      \in
      \left(
      M\left(\cC(\bG)\underline{\otimes}\cK(H)\right)
      ,\ \text{strict topology}
      \right)
      &\overset{}{\cong}
        \cC_b(\bG,\cL(H)_{\text{strict}})
        \quad\left(\text{\cite[Corollary 3.4]{apt}}\right)\\
         &:=
           \left\{
           \bG\xrightarrow[\ \text{strictly continuous}\ ]{\ \text{bounded}}\cL(H)\cong M\left(\cK(H)\right)
        \right\}
    \end{aligned}    
  \end{equation}
  all have size-2 spectrum $\{1,\alpha\}$ and converge strictly to the singleton-spectrum trivial representation on $H$: indeed, the strict topology on $\cL(H)$ is Mackey \cite[Corollary 2.8]{taylor}, so coincides with the strong$^*$ topology on bounded sets by \cite[Theorem II.7]{zbMATH03253462} again (cf. \Cref{res:wtops}\Cref{item:res:wtops:tops}).
\end{example}

\begin{remark}\label{re:bdd.strct}
  $\bG$ being compact, the boundedness requirement in \Cref{eq:ulambda} is not needed (i.e. it is fulfilled automatically):
  \begin{itemize}[wide]
  \item compact subsets of topological vector spaces are \cite[\S 15.6(6)]{k_tvs-1} \emph{bounded} in the abstract sense of \cite[\S 15.6]{k_tvs-1};

  \item and for the strict topology on $\cL(H)\cong M\left(\cK(H)\right)$ that abstract boundedness coincides with the usual notion of norm boundedness per the \emph{uniform boundedness principle} \cite[\S 15.13(2)]{k_tvs-1}.
  \end{itemize}
\end{remark}

\begin{theorem}\label{th:unifmor2amfld}
  Let $\bG$ be a compact quantum group and $A$ a unital $C^*$-algebra.

  \begin{enumerate}[(1),wide]
  \item The finite-spectrum representations $U\in \cC(\bG)\underline{\otimes}A$, in their norm topology, constitute a real analytic Banach submanifold
    \begin{equation}\label{eq:th:unifmor2amfld:repfs}
      \cat{Rep}_{\cat{fs}}(\bG,A)
      \subset
      \text{unitary group }
      U(\cC(\bG)\underline{\otimes}A).
    \end{equation}

  \item The manifold \Cref{eq:th:unifmor2amfld:repfs} is locally homogeneous under the conjugation action of the unitary group $U(A)$, in the sense that the orbits of that action are open in $\cat{Rep}_{\cat{fs}}$.

  \item For every $U(A)$-orbit $\cO_{U}$ in $\cat{Rep}_{\cat{fs}}(\bG,A)$ the orbit map
    \begin{equation*}
      U(A)\ni V
      \xmapsto{\quad}
      V\triangleright U := (1\otimes V)U(1\otimes V^*)
      \in \cO_U
    \end{equation*}
    is an analytic principal fibration.
    
  \item In particular, for every $U\in \cat{Rep}_{\cat{fs}}(\bG,A)$ there is an open neighborhood $\cU\ni U$ and an analytic map
    \begin{equation*}
      \cU\ni U'
      \xmapsto{\quad}
      V_{U'}\in U(A)
      ,\quad
      V_U=1
      ,\quad
      V_{U'}\triangleright U'=U. 
    \end{equation*}
  \end{enumerate}
\end{theorem}

A point of clarification is in order, concerning an ambiguous term: \emph{submanifolds} $N\subseteq M$ are understood here as in \cite[\S II.2]{lang-fund} (the \emph{sous-vari\'et\'es} of \cite[\S 5.8.1]{bourb_vars_1-7}): the Banach tangent-space embeddings $T_pN\le T_pM$, $p\in N$ are required to split, i.e. be \emph{complemented} \cite[\S 4.9]{nb_tvs}. Or: what \cite[Definition 3.1.10]{2602.12362v1} refers to as \emph{split} submanifolds. 

\begin{proof}
  Embed $A$ into some $\cL(H)$, so that $\bG$-representations \emph{in} $A$ are also unitary $\bG$-representations \emph{on} $H$. Given a finite-spectrum $\bG$-representation on $H$, all others sufficiently close to it will have the same (finite) spectrum by \Cref{le:spc.cont}\Cref{item:le:spc.cont:a}. It follows that \Cref{eq:th:unifmor2amfld:repfs} decomposes into open disjoint subsets
  \begin{equation*}
    \cat{Rep}_{\cF}(\bG,A)
    :=
    \left\{\rho\in \cat{Rep}_{\cat{fs}}(\bG,A)\ |\ \spec(\rho)=\cF\right\}
    ,\quad
    \cF\subseteq \mathrm{Irr}(\bG)\text{ finite}.
  \end{equation*}
  Note next that by \Cref{le:nmnm} below
  \begin{equation}\label{eq:rep.le.f}
    \cat{Rep}_{\subseteq \cF}(\bG,A)
    :=
    \cup_{\cF'\subseteq \cF}\cat{Rep}_{\cF'}(\bG,A)
  \end{equation}
  is topologically identifiable with the space
  \begin{equation*}
    \left\{\text{unital morphisms }\cC(\bG)^*_{\cF}\xrightarrow{\quad}A\right\},
  \end{equation*}
  where $\cC(\bG)^*_{\cF}$ is the $C^*$-algebra dual to the subcoalgebra
  \begin{equation}\label{eq:cfg}
    \cC(\bG)_{\cF}
    :=
    \bigoplus_{\alpha\in \cF}\spn\{u^{\alpha}_{ij}\ |\ 1\le i,j\le \dim\alpha\}
    \subseteq \cC(\bG)
  \end{equation}  
  The conclusion now follows in its entirety from the analogous statement for spaces of representations of finite-dimensional $C^*$-algebras (\cite[Theorem 4.1]{zbMATH07787455} for instance), with one caveat: the individual open subspaces
  \begin{equation*}
    \cat{Rep}_{\subseteq \cF}(\bG,A)
    \subseteq
    \cat{Rep}_{\cat{fs}}(\bG,A)
  \end{equation*}
  are realized as submanifolds of the respective Banach spaces $\cL\left(\cC(\bG)^*_{\cF},A\right)$ rather than the desired ambient $U(\cC(\bG)\underline{\otimes}A)$. The following device will supply the deficiency.

  \begin{itemize}[wide]
  \item The embedding
    \begin{equation*}
      U(\cC(\bG)\underline{\otimes}A)
      \le
      (\cC(\bG)\underline{\otimes}A)^{\times}
      :=
      \left\{\text{invertible elements in the displayed $C^*$-algebra}\right\}
    \end{equation*}
    being a submanifold \cite[Corollary 15.22]{upm_ban} , it suffices to confirm a manifold embedding
    \begin{equation}\label{eq:cfa.times}
      \cat{Rep}_{\subseteq \cF}(\bG,A)
      \subseteq
      (\cC(\bG)\underline{\otimes}A)^{\times}
      ,\quad
      \text{fixed finite }\cF\in 2^{\mathrm{Irr}(\bG)}.
    \end{equation}
  \item Now restrict attention to the (complemented) subspace
    \begin{equation*}
      \cL\left(\cC(\bG)_{\cF}^*,A\right)
      \cong 
      \cC(\bG)_{\cF}\otimes A
      =
      \im \bigoplus_{\alpha\in \cF}\left(P_{\alpha}\in \cL\cC(\bG)\right)
      \le
      \cC(\bG)\underline{\otimes}A
      \quad
      \left(\text{see \Cref{eq:cfg}}\right).
    \end{equation*}

  \item The chain
    \begin{equation*}
      C^*\left(\cC(\bG)_{\cF}^*,A\right)
      \le
      \left(\text{unital-Banach-algebra morphisms $\cC(\bG)_{\cF}^*\to A$}\right)
      \le
      \cL\left(\cC(\bG)_{\cF}^*,A\right)
    \end{equation*}
    is now one of analytic Banach submanifolds by \cite[Theorems 4.1 and 4.4]{zbMATH07787455}.

  \item And finally, the manifold embedding \Cref{eq:cfa.times} follows by intersecting with the \emph{open} submanifold of invertible elements in $\cC(\bG)\underline{\otimes}A$.
  \end{itemize}
\end{proof}

In particular, \Cref{th:isautofinsp,th:unifmor2amfld} jointly produce

\begin{corollary}\label{cor:if.cls.pt}
  The conclusion of \Cref{th:unifmor2amfld} holds for
  \begin{equation*}
    \cat{Rep}(\bG,A)
    =
    \cat{Rep}_{\cat{fs}}(\bG,A)
    \subset
    U(\cC(\bG)\underline{\otimes}A)
  \end{equation*}
  for any unital $C^*$-algebra $A$ and compact quantum group $\bG$ with a classical point.  \qedhere
\end{corollary}

We remind the reader that for compact quantum $\bG$ and a unital $C^*$-algebra $A$ (and in fact in much broader generality \cite[Propositions 5.2 and 5.3]{kus-univ}) there is a bijection
\begin{equation}\label{eq:can}
  \cat{Mor}\left(c_0(\widehat{\bG}),A\right)
  \xrightarrow[\quad\cong\quad]{\quad\cat{can}=\cat{can}_{\bG,A}\quad}
  \cat{Rep}(\bG,A),
\end{equation}
where
\begin{itemize}[wide]
\item $\cat{Mor}(B,B')$ denotes the space of \emph{non-degenerate} $C^*$ morphisms $B\xrightarrow{\psi} M(B')$ for possibly non-unital $B,B'$, i.e. \cite[p.291]{kus-univ} with $\overline{\spn \psi(B)B'}=B'$;

\item and
  \begin{equation*}
    c_0(\widehat{\bG})
    :=
    \left\{(x_{\alpha})_{\alpha\in \mathrm{Irr}(\bG)}\in \prod_{\mathrm{Irr}(\bG)}C^{\alpha *}\ :\ \lim_{\substack{\alpha\text{ away}\\\text{from finite sets}}}\|x_{\alpha}\|=0\right\}
  \end{equation*}
\end{itemize}
is the $C^*$-algebra of functions vanishing at infinity on the (discrete quantum) \emph{Pontryagin dual} \cite[\S 2.3]{tim} of $\bG$. 

\begin{lemma}\label{le:nmnm}
  For compact quantum $\bG$ the bijection \Cref{eq:can} induces homeomorphisms
  \begin{equation*}
    \forall\left(\text{finite }\cF\subseteq \mathrm{Irr}(\bG)\right)
    \ :\ 
    \cat{Mor}\left(\prod_{\alpha\in \cF}C^{\alpha*},A\right)
    \xrightarrow[\quad\cong\quad]{\quad\cat{can}=\cat{can}_{\bG,A}\quad}
    \cat{Rep}_{\subseteq \cF}(\bG,A)
  \end{equation*}
  for the norm topologies on the (co)domain.
\end{lemma}
\begin{proof}
  This will be immediate once the explicit constructions for $\cat{can}^{\pm 1}$ are recalled. In one direction,
  \begin{equation*}
    \cat{Mor}\left(c_0(\widehat{\bG}),A\right)
    \ni
    \psi
    \xmapsto{\quad\cat{can}\quad}
    U_{\psi}
    :=
    (\id\otimes \psi)W
    \in
    \cat{Rep}(\bG,A)
  \end{equation*}
  (indeed norm-norm continuous) for the canonical representation \cite[Proposition 5.1]{kus-univ}
  \begin{equation*}
    W
    \in
    \cat{Rep}(\bG,c_0(\widehat{\bG}))
    \subseteq
    UM\left(\cC_r(\bG)\underline{\otimes}c_0(\widehat{\bG})\right).
  \end{equation*}
  Backwards,
  \begin{equation*}
    \cat{Rep}(\bG,A)
    \ni
    U
    \xmapsto{\quad\cat{can}^{-1}\quad}
    \psi_U
    :=
    \left(c_{00}(\widehat{\bG})\ni \varphi\xmapsto{\quad}(\varphi\otimes\id)U\right)
    \in
    \cat{Mor}\left(c_0(\widehat{\bG}),A\right)
  \end{equation*}
  with
  \begin{equation*}
    c_{00}(\widehat{\bG})
    :=
    \bigoplus_{\alpha\in \mathrm{Irr}(\bG)}C^{\alpha *}
    \underset{\text{dense ideal}}{\trianglelefteq}
    c_{0}(\widehat{\bG})
  \end{equation*}
  regarded as a space of functionals on $\cC_r(\bG)$. The conclusion follows upon restricting to the finite-dimensional $\prod_{\alpha\in \cF}C^{\alpha*}$, contained in $c_{00}(\widehat{\bG})$.
\end{proof}

Recall also the discussion in \cite[\S 3.1]{Chirvasitu2026JNCG604} of $\bG$-representations (or actions) on Banach spaces $E$ generally (as opposed to unitary representations on Hilbert spaces). Such a gadget is cast in \cite[Definition 3.1]{Chirvasitu2026JNCG604} as
\begin{itemize}[wide]
\item a continuous linear
  \begin{equation}\label{eq:rep.on.e}
    E
    \xrightarrow{\quad\rho\quad}
    \cC(\bG)\otimes_{\varepsilon}E
    \quad
    \left(\text{\emph{injective} Banach-space tensor product \cite[p.43]{dfs_tens-1}}\right),
  \end{equation}
  coassociative in the sense that $(\Delta\otimes \id)\rho=(\id\otimes \rho)\rho$;
  
\item with
  \begin{equation*}
    \overline{
      \left\{
        (x\otimes \id)\rho v
        \ :\
        x\in \cC(\bG),\ v\in E
      \right\}
    }
    =
    \cC(\bG)\otimes_{\varepsilon}E. 
  \end{equation*}
\end{itemize}
The spectral machinery goes through for Banach space representations, much as in the argument underpinning the more familiar context \cite[Theorem 1.5]{podl_symm} of $\bG$-actions on $C^*$-algebras: there are idempotents
\begin{equation*}
  P^{\alpha}:=\left(\psi^{\alpha}\otimes\id\right)\rho
  \in \cL(E)
\end{equation*}
cutting out isotypic components $E^{\alpha}$, one then defines the spectrum $\spec\rho$ in the obvious fashion, etc. 

\Cref{th:unifmor2amfld} now has its point-by-point Banach-space-representation counterpart in the following guise. 

\begin{theorem}\label{th:fin.spc.ban}
  Let $\bG$ be a compact quantum group and $E$ a Banach space.
  \begin{enumerate}[(1),wide]
  \item The finite-spectrum representations \Cref{eq:rep.on.e}, in their norm topology, constitute a real analytic Banach submanifold
    \begin{equation*}
      \cat{Rep}_{\cat{fs}}(\bG,E)
      \subset
      \cL(E,\cC(\bG)\otimes_{\varepsilon}E).
    \end{equation*}

  \item That manifold is locally homogeneous under the conjugation action of the general linear group $GL(E)$, in the sense that the orbits of that action are open in $\cat{Rep}_{\cat{fs}}$.

  \item For every $GL(E)$-orbit $\cO_{\rho}$ in $\cat{Rep}_{\cat{fs}}(\bG,E)$ the orbit map
    \begin{equation*}
      GL(E)\ni V
      \xmapsto{\quad}
      V\triangleright \rho := (1\otimes V)\rho(1\otimes V^*)
      \in \cO_{\rho}
    \end{equation*}
    is an analytic principal fibration.
    
  \item In particular, for every $\rho\in \cat{Rep}_{\cat{fs}}(\bG,E)$ there is an open neighborhood $\cU\ni \rho$ and an analytic map
    \begin{equation*}
      \cU\ni \rho'
      \xmapsto{\quad}
      V_{\rho'}\in GL(E)
      ,\quad
      V_{\rho}=1
      ,\quad
      V_{\rho'}\triangleright \rho'=\rho. 
    \end{equation*}
  \end{enumerate}
\end{theorem}
\begin{proof}
  Much of the scaffolding of \Cref{th:unifmor2amfld}'s proof can stay in place, this time shifting reliance to the Banach analogue of the $C^*$-flavored \cite[Theorem 4.1]{zbMATH07787455}. 

  $\bG$-representations on $E$ sufficiently close to $\rho\in \cat{Rep}_{\cat{fs}}(\bG,E)$ will again have spectra equal to $\cF:=\spec\rho$ for reasons entirely parallel to those invoked in the proof of \Cref{th:unifmor2amfld} (per \Cref{le:spc.cont.ban}), so it suffices to focus attention on
  \begin{equation*}
    \cat{Rep}_{\subseteq \cF}(\bG,E)
    \subseteq
    \cat{Rep}_{\cat{fs}}(\bG,E)
  \end{equation*}
  defined by analogy with \Cref{eq:rep.le.f}. Now identify
  \begin{equation*}
    \begin{aligned}
      \cat{Rep}_{\subseteq \cF}(\bG,E)
      &\cong
        \left\{\text{$\cC(\bG)_{\cF}$-comodule structures on $E$}\right\}\\
      &\cong
        \left\{\text{$\cC(\bG)^*_{\cF}$-module structures on $E$}\right\}
        \quad\left(\text{\cite[Lemma 1.6.4]{mont}}\right)\\
      &\cong
        \cat{BAlg}_1(\cC(\bG)^*_{\cF}, \cL(E))
        :=
        \left\{\text{unital Banach-algebra morphisms }\cC(\bG)^*_{\cF}\to \cL(E)\right\}\\
      &\subseteq \cL\left(\cC(\bG)^*_{\cF}, \cL(E)\right)
        \quad
        \left(\text{Banach submanifold: \cite[Theorem 4.4]{zbMATH07787455}}\right)\\
      &\cong
        \cL\left(E,\cC(\bG)_{\cF}\otimes E\right)\\
      &\cong
        \cL\left(E,\cC(\bG)\otimes_{\varepsilon} E\right)
      \quad\left(\text{split linear, hence submanifold}\right).
    \end{aligned}    
  \end{equation*}
  The claims again all follow from their analogues \cite[Theorem 4.4]{zbMATH07787455} for submanifolds $\cat{BAlg}_1(A,B)\le \cL(A,B)$ for unital Banach algebras $A,B$ with $A$ finite-dimensional semisimple (as $\cC(\bG)^*_{\cF}$ certainly is). 
\end{proof}

\begin{lemma}\label{le:spc.cont.ban}
  Let $\bG$ be a compact quantum group and $E$ a Banach space. 
  \begin{enumerate}[(1),wide]
  \item\label{item:le:spc.cont.ban:all} The map
    \begin{equation*}
      \cat{Rep}(\bG,E)
      \ni \rho
      \xmapsto{\quad}
      \spec \rho
      \in 2^{\mathrm{Irr}(\bG)}=\{0,1\}^{\mathrm{Irr}(\bG)}.
    \end{equation*}
    is continuous for the uniform topology on the domain and the compact-open topology on the codomain.

  \item\label{item:le:spc.cont.ban:fin} In particular, that map's restriction to the open subset
    \begin{equation*}
      \cat{Rep}_{\cat{fs}}(\bG,E)
      \subseteq
      \cat{Rep}(\bG,E)
    \end{equation*}
    is locally constant.
  \end{enumerate}
\end{lemma}
\begin{proof}
  The argument runs parallel to its Hilbert-space antecedent: for each $\alpha\in \mathrm{Irr}(\bG)$ the map
  \begin{equation*}
    \cat{Rep}(\bG,E)
    \ni \rho
    \xmapsto{\quad}
    P_{\rho}^{\alpha}=\left(\psi^{\alpha}\otimes\id\right)\rho
    \in
    \cL(E)
  \end{equation*}
  is norm-norm continuous, so for finite $\cF\in 2^{\mathrm{Irr}(\bG)}$ sufficiently close tuples $\left(P^{\alpha}_{\rho}\right)_{\alpha\in \cF}$ and $\left(P^{\alpha}_{\rho'}\right)_{\alpha\in \cF}$ will be mutual $GL(E)$ conjugates (\cite[Theorem 8.2.3]{rnd_amnbl_2002} or \cite[Theorem 4.4]{zbMATH07787455}). In particular, one such tuple will sum to $1\in \cL(E)$ if and only if the other does. 
\end{proof}


\addcontentsline{toc}{section}{References}



\Addresses

\end{document}